\documentclass[12pt]{amsart}
\usepackage{amssymb,amsmath,amsfonts,amsthm,bbm,mathrsfs,times,graphicx,color,comment}
\usepackage{tikz}
\usetikzlibrary{arrows,shapes}
\usetikzlibrary{fadings}
\usetikzlibrary{intersections}
\usetikzlibrary{decorations.pathreplacing}

\DeclareMathOperator{\re}{\mathbb{R}e}
\DeclareMathOperator{\im}{\mathbb{I}m}

\newcommand{\gm}{{\Gamma_-}}
\newcommand{\gp}{{\Gamma_+}}

\newcommand{\INF}{{\infty}}
\newcommand{\eps}{\epsilon}

\newcommand{\tta}{\theta}

\newcommand{\OM}{\Omega}

\newcommand{\sph}{{{\mathbf S}^ 1}}

\newcommand{\del}{\partial}

\newcommand{\Gam}{\varGamma}
\newcommand{\ol}{\overline}
\newcommand{\ds}{\displaystyle}

\newcommand{\BR}{\mathbb{R}}

\newcommand{\fii}{{\varphi}}
\newcommand{\bu}{{\bf u}}
\newcommand{\bv}{{\bf v}}

\newcommand{\bg}{{\bf g}}

\newcommand{\bF}{{\bf F}}
\newcommand{\bI}{{\bf I}}

\newcommand{\B}{\mathcal{B}}

\newcommand{\HT}{\mathcal{H}}

\newtheorem{theorem}{Theorem}[section]
\newtheorem{prop}{Proposition}[section]
\newtheorem{lemma}{Lemma}[section]
\newtheorem{cor}{Corollary}[section]

\newtheorem{definition}{Definition}[section]

\title[On the $X$-ray transform of symmetric 2-tensors]{On the $X$-ray transform of planar symmetric 2-tensors}
\begin{document}
\date{\today}
\author{Kamran Sadiq}
\address{Johann Radon Institute of Computational and Applied Mathematics (RICAM), Altenbergerstrasse 69, 4040 Linz, Austria}
\email{kamran.sadiq@oeaw.ac.at}

\author{Otmar Scherzer}
\address{Computational Science Center, Oskar-Morgenstern-Platz 1, 1090 Vienna \& Johann Radon Institute of Computational and Applied Mathematics (RICAM), Altenbergerstrasse 69, 4040 Linz, Austria}
\email{otmar.scherzer@univie.ac.at}

\author{Alexandru Tamasan}
\address{Department of Mathematics, University of Central Florida, Orlando, 32816 Florida, USA}
\email{tamasan@math.ucf.edu}

\maketitle

\begin{abstract}
In this paper we study the attenuated $X$-ray transform of 2-tensors 
supported in strictly convex bounded subsets in the Euclidean plane. We characterize its range
and reconstruct all possible 2-tensors yielding identical $X$-ray data. 
The characterization is in terms of a Hilbert-transform associated with $A$-analytic maps
in the sense of Bukhgeim.
\end{abstract}

\section{Introduction} \label{S:intro}
This paper concerns the range characterization of the attenuated $X$-ray transform of symmetric 2-tensors in the plane. 
Range characterization of the non-attenuated $X$-ray transform of functions (0-tensors) in the
Euclidean space has been long known \cite{gelfandGraev, helgason,ludwig}, whereas in the case of a constant attenuation some range conditions can be inferred from \cite{kuchmentLvin,aguilarKuchment,aguilarEhrenpreisKuchment}. For a varying attenuation the two dimensional case has been particularly interesting with inversion formulas requiring new analytical tools: the theory of $A$-analytic maps originally employed in \cite{ABK}, and ideas from inverse scattering  in \cite{novikov01}. Constraints on the range for the two dimensional $X$-ray transform of functions were given in \cite{novikov02,bal}, and a range characterization based on Bukhgeim's theory of $A$-analytic maps was given in \cite{sadiqtamasan01}.

Inversion of the $X$-ray transform of higher order tensors has been formulated 
directly in the setting of Riemmanian manifolds with boundary \cite{vladimirBook}. 
The case of 2-tensors appears in the linearization of the boundary rigidity
problem. It is easy to see that injectivity can hold only in some restricted 
class: e.g., the class of solenoidal tensors.
For two dimensional simple manifolds with boundary, injectivity with in the solenoidal tensor fields
has been establish fairly recent: in the non-attenuated case for 0- and 1-tensors we mention 
the breakthrough result in \cite{pestovUhlmann04}, and in the attenuated case in \cite{saloUhlmann11};
see also \cite{holmanStefanov} for a more general weighted transform. 
Inversion for the attenuated $X$-ray transform for solenoidal tensors of rank two and higher can be 
found in \cite{paternainSaloUhlmann13}, with a range characterization in \cite{paternainSaloUhlmann14}.
In the Euclidean case we mention an earlier inversion of the attenuated $X$-ray transform of solenoidal 
tensors in \cite{kazantsevBukhgeimJr06}; however this work does not address range characterization.

Different from the recent characterization in terms of the scattering relation in \cite{paternainSaloUhlmann14}, in this paper the range conditions are in terms of the Hilbert-transform for $A$-analytic
maps introduced in \cite{sadiqtamasan01, sadiqtamasan02}. 
Our characterization can be understood as an explicit description of the scattering relation in 
\cite{paternainSaloUhlmann13-1,paternainSaloUhlmann13, paternainSaloUhlmann14} particularized to the Euclidean setting. In the sufficiency part we reconstruct all possible 2-tensors yielding  identical $X$-ray data; see \eqref{RT_PsiClass} for the non-attenuated case and  \eqref{PsiClass} for the attenuated case.

For a real symmetric 2-tensor  $\bF\in L^1(\BR^2;\BR^{2\times2})$,  
\begin{align}\label{defn_F}
\bF(x)=\begin{pmatrix} f_{11}(x) & f_{12}(x) \\ f_{12}(x) & f_{22}(x) \end{pmatrix}, \quad x \in \BR^2,
\end{align}
and a real valued function $a\in L^1(\BR^2)$,  the \emph{$a$-attenuated $X$-ray transform  of $\bF$} 
is defined by
\begin{align}\label{XaF}
X_{a}\bF(x,\tta) := \int_{-\infty}^{\infty} \langle \bF(x+t\tta)\, \tta , \tta \rangle \exp\left\{{-\int_t^\infty a(x+s\tta)ds}\right\}dt,
\end{align}
where $\tta$ is a direction in the unit sphere $\sph$, and $\langle\cdot,\cdot \rangle$ is the scalar product in $\BR^2$. For the non attenuated case $a\equiv 0$ we use the notation $X\bF$.

In this paper, we consider $\bF$ be defined on a strictly convex bounded set $\Omega \subset \BR^2$ with vanishing trace at the boundary $\Gamma$; further regularity and the order of vanishing will be specified in the theorems. In the attenuated case we assume $a > 0$ in $\ol{\OM}$.

For any $(x,\theta)\in\ol\OM\times \sph$ let $\tau(x,\theta)$ 
be  length of the chord in the direction of $\theta$ passing through $x$. 
Let  also consider the incoming $(-)$, respectively outgoing $(+)$ submanifolds of the unit bundle restricted to the boundary
\begin{align}\label{GammaPM}
\Gamma_\pm :=\{(x,\tta)\in \Gamma\times \sph: \pm\tta\cdot n(x)>0\},
\end{align} and the variety 
\begin{align}\label{GammaZero}
\Gamma_0 :=\{(x,\tta)\in \Gamma\times \sph: \tta\cdot n(x)=0\},
\end{align} where $n(x)$ denotes outer normal.

\begin{figure}[ht]
\centering
\begin{tikzpicture}[scale=1.7,cap=round,>=latex]

\tikzset{
    thick/.style=      {line width=0.8pt},
    very thick/.style= {line width=1.1pt},
    ultra thick/.style={line width=1.6pt}
}

\coordinate[label=above:$\OM$] (OM) at(0:0.1cm);

 \draw[->] (45:1.2cm) -- (45:1.6cm);
  \coordinate[label=above:$\theta$] (theta1) at (45:1.2cm); 
 
  \filldraw[black] (0:1cm) circle(1.2pt);
  	\draw[->] (0:1cm) -- (0:1.6cm);
  	\coordinate[label=left:$n$] (nu1) at (0:1.9cm);

 \filldraw[black] (70:1cm) circle(1.2pt);
  	\draw[->] (70:1cm) -- (70:1.6cm);
    \coordinate[label=above:$n$] (nu2) at (70:1.5cm);

  \coordinate[label=right:$\gp$] (gp) at (25:1.01cm);   
  	\draw[ultra thick] (135:0.8cm) -- (135:1.2cm);
  	\draw[thick,draw={red}](315:1cm)  arc (315:135:1.01cm);

   \draw[->] (225:1.6cm) -- (225:1.2cm);
  \coordinate[label=left:$\theta$] (theta2) at (225:1.2cm); 
 
  \filldraw[black] (170:1cm) circle(1.2pt);
  	\draw[->] (170:1cm) -- (170:1.6cm);
 	\coordinate[label=left:$n$] (nu3) at (170:1.5cm);

  \filldraw[black] (250:1cm) circle(1.2pt); 
  	\draw[->] (250:1cm) -- (250:1.6cm);
 	\coordinate[label=below:$n$] (nu4) at (250:1.6cm);

  \coordinate[label=left:$\gm$] (gm) at at (210:1.01cm);
  	\draw[ultra thick] (315:0.8cm) -- (315:1.2cm); 
  	\draw[ultra thick, draw={blue}](315:1cm)  arc (-45:135:1.01cm); 
    
\end{tikzpicture}
\caption{Definition of $\Gam_\pm$ } \label{fig:1}
\end{figure}
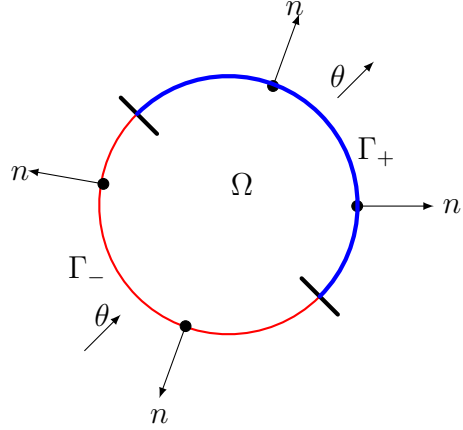

The $a$-attenuated $X$-ray transform of $\bF$ 
is realized as a function on $\Gam_+$ by
\begin{align}\label{ART2T}
X_{a}\bF(x,\tta) = \int_{-\tau(x,\tta)}^{0} \langle \bF(x+t\tta)\, \tta , \tta \rangle \,
e^{-\int_t^0 a(x+s\tta)ds} \;dt, \; (x,\tta) \in \Gam_{+}.
\end{align} 
 
We approach the range characterization through its connection with the transport model as follows: 
The boundary value problem
\begin{align} \label{TransportEq2Tensor}
&  \tta\cdot\nabla u(x,\tta) +a(x) u(x,\theta) = \langle  \bF(x) \tta , \tta \rangle  \quad  (x,\tta)\in \OM \times \sph, \\ \label{UGamMinus}
&   u|_{\Gam_-}= 0 
\end{align}
has a unique solution in $\OM \times \sph$ and
\begin{align}\label{u_Gam+}
u \lvert_{\Gam_{+}} (x,\tta) = X_a \bF (x,\tta), \quad (x,\tta)\in \Gam_{+}.
\end{align}

The $X$-ray transform of 2-tensors occurs in the linearization of the boundary rigidity
problem \cite{vladimirBook}:
For $\eps>0$ small, let 
\begin{align*}
g^\eps(x):=\bI+\eps \bF(x)+ o(\eps), \;x \in \OM,
\end{align*} be a family of metrics perturbations from the Euclidean, 
where $\bI$ is the identity matrix and $\bF$ is as in \eqref{defn_F}. 
For an arbitrary pair of boundary points $x,y\in\Gam$  let 
$d_{\eps}(x,y)$ denote their distance in the metric $g^\eps$.  The boundary rigidity problem 
asks for the recovery of the metric $g^\eps$ from knowledge of $d_\eps(x,y)$ for all $x,y\in\Gam$. 
In the linearized case one seeks to recover $\bF(x)$ from 
$\left.\frac{d}{d\eps}\right|_{\eps=0}d^2_\eps(x,y).$
Taking into account the length minimizing property of geodesic one can show that 
\begin{align*}
\frac{1}{|x-y|}\left.\frac{d}{d\eps}\right|_{\eps=0}d^2_\eps(x,y)
=  \int_{-|x-y|}^0\langle \bF(x+t\tta)\tta,\tta\rangle dt = X\bF(x,\tta),
\end{align*} where $\tta:=\ds \frac{x-y}{|x-y|}\in\sph.$

\section{Preliminaries}
In this section we briefly introduce the properties of 
Bukhgeim's $A$-analytic maps \cite{bukhgeim_book}  needed later. 

For $z=x_1+ix_2$, we consider the Cauchy-Riemann operators
\begin{align} \label{CauchyRiemannOp}
\ol{\del} = \left( \del_{x_{1}}+i \del_{x_{2}} \right) /2 ,\quad \del = \left( \del_{x_{1}}- i \del_{x_{2}} \right) /2.
\end{align} Let $l_\INF (,l_1)$ be the space of bounded (, respectively summable) sequences,
$\mathcal{L}:l_\infty\to l_\infty$ be the left shift
\begin{align*}
\mathcal{L}  \langle u_{-1}, u_{-2}, ... \rangle =  \langle u_{-2}, u_{-3}, u_{-4}, ... \rangle.
\end{align*}
\begin{definition}
A sequence valued map
$$z\mapsto  \bu(z): = \langle u_{-1}(z),u_{-2}(z),u_{-3}(z),... \rangle$$
is called {\em $\mathcal{L}$-analytic}, if $\bu\in C(\ol\OM;l_\INF)\cap C^1(\OM;l_\INF)$ and
\begin{equation}\label{Aanalytic}
\ol{\del} \bu (z) + \mathcal{L} \del \bu (z) = 0,\quad z\in\OM.
\end{equation}
\end{definition}

For $0<\alpha<1$ and $k=1,2$, we recall the Banach spaces in \cite{sadiqtamasan01}:
\begin{equation} \label{lGamdefn}
 l^{1,k}_{\INF}(\Gam): = \left \{ \bu=  \langle u_{-1}, u_{-2}, ... \rangle  : \sup_{\zeta\in \Gam}\sum_{j=1}^{\INF}  j^{k} \lvert u_{-j}(\zeta) \rvert < \INF \right \},
\end{equation}
\begin{equation} \label{CepsGamdefn}
 C^{\alpha}(\Gam ; l_1) := \left \{ \bu:
\sup_{\xi\in \Gam} \lVert \bu(\xi)\rVert_{\ds l_{1}} + \underset{{\substack{
            \xi,\eta \in \Gam \\
            \xi\neq \eta } }}{\sup}
 \frac{\lVert \bu(\xi) - \bu(\eta)\rVert_{\ds l_{1}}}{|\xi - \eta|^{ \alpha}} < \INF \right \}.
\end{equation}
By replacing $\Gam$ with $\overline\OM$ and $l_{1}$ with $l_{\INF}$ in \eqref{CepsGamdefn} we similarly define  $C^{\alpha}(\overline\OM ; l_1)$, respectively, $C^{\alpha}(\ol\OM ; l_{\INF})$.

At the heart of the theory of $A$-analytic maps lies a Cauchy-like integral formula introduced by Bukhgeim in \cite{bukhgeim_book}. The explicit variant \eqref{BukhgeimCauchyFormula} appeared first in Finch \cite{finch}.
The formula  below is restated in terms of $\mathcal{L}$-analytic maps as in \cite{sadiqtamasan02}.

\begin{theorem}\cite[Theorem 2.1]{sadiqtamasan02}\label{BukhgeimCauchyThm}
For some $\bg=\langle g_{-1}, g_{-2},g_{-3},...\rangle  \in l^{1,1}_{\INF}(\Gam)\cap C^\alpha(\Gam;l_1)$ define the Bukhgeim-Cauchy operator $\B $ acting on $\bg$, $$\OM\ni z\mapsto \langle (\B \bg)_{-1}(z) ,(\B \bg)_{-2}(z),(\B \bg)_{-3}(z),...\rangle,$$ by
\begin{align}
(\B \bg)_{-n}(z) &:= \frac{1}{2\pi i}\sum_{j=0}^{\infty} \int_{\Gam}
\frac{ g_{-n-j}(\zeta)\ol{(\zeta-z)}^{j}}{(\zeta-z)^{j+1}}d\zeta
\nonumber \\
 \label{BukhgeimCauchyFormula}
 & \qquad - \frac{1}{2\pi i}\sum_{j=1}^{\INF} \int_{\Gam}
\frac{ g_{-n-j}(\zeta)\ol{(\zeta-z)}^{j-1}}{(\zeta-z)^{j}}d\ol{\zeta}, \; n=1,2,3,...
\end{align}
Then $\B \bg\in C^{1,\alpha}(\OM;l_\infty)\cap C(\ol \OM;l_\infty)$ and it is also $\mathcal{L}$-analytic.
\end{theorem}

For our purposes further regularity in $\B\bg$ will be required. Such smoothness is obtained by increasing the assumptions on the rate of decay of the terms in $\bg$ as explicit below.
For $0<\alpha<1$, let us recall the Banach space $Y_\alpha$ in \cite{sadiqtamasan01}:
\begin{equation} \label{YGamdefn}
 Y_{\alpha} = \left \{ \bg\in l^{1,2}_{\INF}(\Gamma) :  \underset{{\substack{
            \xi, \mu  \in \Gam \\
            \xi \neq \mu } }}{\sup} \sum_{j=1}^{\INF} j\frac{\lvert g_{-j}(\xi) - g_{-j}(\mu) \rvert}{|\xi - \mu|^\alpha} < \INF \right \}.
\end{equation}

\begin{prop}\cite[Proposition 2.1]{sadiqtamasan02}\label{extraRegularity}
If $\bg\in Y_{\alpha}$, $\alpha>1/2$, then 
\begin{align}
\B \bg\in C^{1,\alpha}(\OM;l_1)\cap C^{\alpha}(\ol \OM;l_1)\cap C^2(\OM;l_\infty).
\end{align}
\end{prop}

The Hilbert transform associated with boundary of $\mathcal{L}$-analytic maps is defined below.
\begin{definition}
For $\bg=\langle g_{-1},g_{-2},g_{-3},...\rangle\in l^{1,1}_{\INF}(\Gam)\cap C^\alpha(\Gam;l_1)$, we define the Hilbert transform $\HT\bg$  componentwise for $n \geq 1$ by
\begin{align}\label{hilbertT}
(\HT\bg)_{-n}(\xi)&=\frac{1}{\pi} \int_{\Gam } \frac{g_{-n}(\zeta)}{\zeta - \xi} d\zeta\nonumber\\
&+\frac{1}{\pi} \int_{\Gam } \left \{ \frac{d\zeta}{\zeta-\xi}-\frac{d \ol{\zeta}}{\ol{\zeta-\xi}} \right \} \sum_{j=1}^{\infty}  g_{-n-j}(\zeta)
\left( \frac{\ol{\zeta-\xi}}{\zeta-\xi} \right) ^{j},\; \xi\in\Gam.
\end{align}
\end{definition}

The following result justifies the name of the transform $\HT$. For its proof we refer to \cite[Theorem 3.2]{sadiqtamasan01}.

\begin{theorem}\label{NecSuf}
For $0<\alpha<1$, let $\bg \in l^{1,1}_{\INF}(\Gamma)\cap C^\alpha(\Gamma;l_1)$.
For $\bg$ to be boundary value of an $\mathcal{L}$-analytic function it is necessary and sufficient that
\begin{equation} \label{NecSufEq}
 (I+i\HT) \bg = {\bf {0}},
 \end{equation} where $\HT$ is as in \eqref{hilbertT}.
\end{theorem}

\section{The non-attenuated case } \label{Range2TensorNonAtten}

In this section we assume $a \equiv 0$. We establish necessary and sufficient conditions for a sufficiently smooth function on $\Gam \times \sph$ to be the $X$-ray data of some sufficiently smooth real valued symmetric 2-tensor $\bF$.
For $\tta = (\cos \fii , \sin \fii) \in \sph$, a calculation shows that
\begin{align} \label{Ftta_Calc}
    \langle \bF(x) \tta, \tta \rangle &= f_{0}(x)+\ol{f_{2}(x)} e^{2i\fii}+ f_{2}(x) e^{-2i\fii},
\end{align} where
\begin{align}\label{f0f2Eq}
f_{0}(x) = \frac{f_{11}(x)+f_{22}(x)}{2}, \;\text{and} \;
f_{2}(x) =  \frac{f_{11}(x)-f_{22}(x)}{4} + i \frac{f_{12}(x)}{2}.
\end{align}

The transport equation in \eqref{TransportEq2Tensor}
becomes \begin{align}\label{2TtransportEQ_a=0}
&\tta\cdot\nabla u(x,\tta)= f_{0}(x)+\ol{f_{2}(x)} e^{2i\fii}+ f_{2}(x) e^{-2i\fii}, \quad x\in\OM.
\end{align}

For $z=x_1+ix_2\in \OM$, we consider the Fourier expansions of $u(z,\cdot)$ in the angular variable $\tta= ( \cos\fii,\sin\fii)$:
\begin{align*}
u(z,\theta) = \sum_{-\infty}^{\infty} u_{n}(z) e^{in\fii}.
\end{align*}Since $u$ is real valued its Fourier modes occur in conjugates,
$$u_{-n}(z)= \ol{u_{n}(z)}, \quad n \geq 0, \;z\in\OM.$$

With the Cauchy-Riemann operators defined in \eqref{CauchyRiemannOp} the advection operator becomes
\begin{align*}
\tta\cdot\nabla = e^{-i\fii} \ol{\del} + e^{i\fii} \del.
\end{align*}
Provided appropriate convergence of the series (given by smoothness in the angular variable)
we see that if $u$ solves \eqref{2TtransportEQ_a=0} then its Fourier modes solve the system
\begin{align} \label{fzeroEq}
&\ol{\del} u_{1}(z) + \del u_{-1}(z) = f_{0}(z), \\   \label{f2Eq}
&\ol{\del} u_{-1}(z) + \del u_{-3}(z) = f_{2}(z),
 \\ \label{NonAttenRT2AnalyticEq_Even}
&\ol{\del} u_{2n}(z) + \del u_{2n-2}(z) = 0, \quad n \leq 0,
 \\ \label{NonAttenRT2AnalyticEq_Odd}
&\ol{\del} u_{2n-1}(z) + \del u_{2n-3}(z) = 0, \quad n \leq -1,
\end{align}

The range characterization is given in terms of the trace
\begin{align}
g:=u \lvert_{\Gam \times \sph} = \left \{ \begin{array}{ll}
X \bF(x,\tta) , & (x,\tta) \in \Gam_{+}, \\
0 , &  (x,\tta) \in \Gam_{-} \cup \Gam_{0}. \\
\end{array}
  \right.
\end{align}
More precisely, in terms of its Fourier modes in the angular variables:
\begin{align}\label{FourierData}
g(\zeta,\theta) = \sum_{-\infty}^{\infty} g_{n}(\zeta) e^{in\fii}, \quad \zeta \in \Gam.
\end{align}Since the trace $g$ is also real valued, its Fourier modes will satisfy
\begin{align}
g_{-n}(\zeta)= \ol{g_{n}(\zeta)}, \quad n \geq 0,\;\zeta\in\Gam.
\end{align}
From the negative even modes, we built the sequence
\begin{align}\label{gEven}
\bg^{even}:=\langle g_{0}, {g}_{-2}, g_{-4}, ... \rangle.
\end{align}
From the negative odd modes starting from mode $-3$, we built the sequence
\begin{align}\label{gOddNoMinus1}
\bg^{odd}:=\langle g_{-3}, {g}_{-5}, g_{-7}, ... \rangle.
\end{align}

Next we characterize the data $g$ in terms of the Hilbert Transform $\HT$ in \eqref{hilbertT}. 
We will construct simultaneously the right hand side of the transport equation \eqref{2TtransportEQ_a=0} and the solution $u$ whose trace matches the boundary data $g$. Construction of $u$ is via its Fourier  modes. We first construct the negative modes and then the positive modes are constructed by conjugation. Except from negative one mode $u_{-1}$ all non-positive modes are defined by  
Bukhgeim-Cauchy integral formula in \eqref{BukhgeimCauchyFormula} using boundary data.  Other then having the trace $g_{-1}$ on the boundary $u_{-1}$ is unconstrained. It is chosen arbitrarily from the class of functions
\begin{align}\label{RT_PsiClass}
\Psi_{g}:=
&\left \{ \vphantom{\int}
\psi \in C^{1}(\ol\OM;\mathbb{C}):\psi \lvert_{\Gam}= g_{-1} \right\}.
\end{align}   

\begin{theorem}[Range characterization in the non-attenuated case]\label{RT2Tensor}

Let $\alpha > 1/2$.

(i) Let $\bF \in C_{0}^{1,\alpha}(\OM;\BR^{2 \times 2})$. For $g := \left \{ \begin{array}{ll}
X \bF(x,\tta) , & (x,\tta) \in \Gam_{+}, \\
0 , &  (x,\tta) \in \Gam_{-} \cup \Gam_{0}, \\
\end{array}
  \right. $ consider the corresponding sequences $\bg^{even}$ as in
\eqref{gEven} and $\bg^{odd}$ as in \eqref{gOddNoMinus1}. Then
$\bg^{even}, \bg^{odd} \in l^{1,1}_{\INF}(\Gamma)\cap C^\alpha(\Gamma;l_1)$ satisfy
\begin{align}\label{RT2TensorCond1}
&[I+i\HT]\bg^{even} = {\bf {0}}, \\ \label{RT2TensorCond2}
&[I+i\HT]\bg^{odd} = {\bf {0}},
\end{align} where the operator $\HT$ is the Hilbert transform in \eqref{hilbertT}.

(ii) Let $g\in C^{\alpha} \left(\Gam; C^{1,\alpha}(\sph) \right)\cap C(\Gam;C^{2,\alpha}(\sph))$ be real valued with $g \lvert_{\Gam_{-} \cup \Gam_{0}}=0$.
If the corresponding sequence $\bg^{even}, \bg^{odd} \in Y_{\alpha}$ satisfies \eqref{RT2TensorCond1} and  \eqref{RT2TensorCond2},
then there exists a real valued symmetric 2-tensor $\bF \in C(\OM;\BR^{2 \times 2})$, such that $g \lvert_{\Gam_{+}} = X\bF$.
Moreover for each $\psi \in \Psi_{g}$ in \eqref{RT_PsiClass}, there is a unique real valued symmetric 2-tensor $\bF_{\psi}$ such that
$g\lvert_{\Gam_{+}} = X\bF_{\psi}$.
\end{theorem}

\begin{proof} (i) {\bf Necessity}

Let $\bF \in C_{0}^{1,\alpha}(\OM;\BR^{2 \times 2})$. Since $\bF$ is compactly supported inside \
$\OM$, for any point at the boundary there is a cone of lines which do not meet the support.
Thus $g \equiv 0$ in the neighborhood of the variety $\Gam_0$
 which yields $g \in C^{1,\alpha}(\Gam \times \sph)$. Moreover, $g$ is the trace on $\Gam \times \sph$ of a solution $u \in C^{1,\alpha}(\ol{\OM} \times \sph)$ of the transport equation \eqref{2TtransportEQ_a=0}.
By \cite[Proposition 4.1]{sadiqtamasan01} $\bg^{even}, \bg^{odd} \in l^{1,1}_{\INF}(\Gamma)\cap C^\alpha(\Gamma;l_1)$. 

If $u$ solves \eqref{2TtransportEQ_a=0} then its Fourier modes satisfy \eqref{fzeroEq}, \eqref{f2Eq}, \eqref{NonAttenRT2AnalyticEq_Even} and \eqref{NonAttenRT2AnalyticEq_Odd}. Since the negative even Fourier modes $u_{2n}$ of $u$ satisfies the system \eqref{NonAttenRT2AnalyticEq_Even} for $n \leq 0$, then
$$z\mapsto  \bu^{even}(z):= \langle u_{0}(z), u_{-2}(z), u_{-4}(z), u_{-6}(z), \cdots \rangle$$
is $\mathcal{L}$-analytic in $\OM$ and the necessity part in Theorem  \ref{NecSuf} yields  \eqref{RT2TensorCond1}.

The equation \eqref{NonAttenRT2AnalyticEq_Odd} for  negative odd Fourier modes $u_{2n-1}$ starting from mode $-3$ yield that the sequence valued map
$$z\mapsto  \bu^{odd}(z):= \langle u_{-3}(z), u_{-5}(z), u_{-7}(z), \cdots \rangle$$
is $\mathcal{L}$-analytic in $\OM$ and the necessity part in Theorem  \ref{NecSuf} yields  \eqref{RT2TensorCond2}.

(ii) {\bf Sufficiency}

To prove the sufficiency we will construct a real valued symmetric 2-tensor $\bF$ in $\OM$ and a real valued function $u \in C^{1}(\OM\times \sph)\cap C(\ol\OM \times \sph)$
such that $u \lvert_{\Gam \times \sph}=g$ and $u$ solves \eqref{2TtransportEQ_a=0} in $\OM$.
The construction of such $u$ is in terms of its Fourier modes in the angular variable and it is done in several steps.

{\bf Step 1: The construction of negative even modes $u_{2n}$ for $n \leq 0$.}

Let $g\in C^{\alpha} \left(\Gam; C^{1,\alpha}(\sph) \right)\cap C(\Gam;C^{2,\alpha}(\sph))$  be real valued with $g \lvert_{\Gam_{-} \cup \Gam_{0}}=0$.
Let the corresponding sequences  $\bg^{even}$  satisfying \eqref{RT2TensorCond1} and $\bg^{odd}$  satisfying \eqref{RT2TensorCond2}.
By \cite[Proposition 4.1(ii)]{sadiqtamasan01} $\bg^{even}, \bg^{odd}\in Y_{\alpha}$.
Use the Bukhgeim-Cauchy Integral formula \eqref{BukhgeimCauchyFormula} to 
construct the negative even Fourier modes:
\begin{align}\label{construction_EVENS}
\langle u_{0}(z), u_{-2}(z), u_{-4}(z), u_{-6}(z), ... \rangle := \B \bg^{even}(z), \quad z\in \OM.
\end{align}
By Theorem \ref{BukhgeimCauchyThm}, the sequence valued map
\begin{align*}
z\mapsto \langle u_{0}(z), u_{-2}(z), u_{-4}(z),  ... \rangle ,
\end{align*}is $\mathcal{L}$-analytic in $\OM$, thus the equations
\begin{align}\label{uEvenL1}
\ol{\del} u_{-2k} + \del u_{-2k-2} = 0,
\end{align} are satisfied for all $k \geq 0$.
Moreover, the hypothesis \eqref{RT2TensorCond1} and the sufficiency part of Theorem \ref{NecSuf} yields that they extend continuously to $\Gam$ and
\begin{align}
u_{-2k}|_\Gam=g_{-2k},\quad k\geq 0.
\end{align}

{\bf Step 2: The construction of positive even modes $u_{2n}$ for $n \geq 1$.}

All of the positive even Fourier modes are constructed by conjugation:
\begin{align}\label{construct_even+}
u_{2k}:=\ol{u_{-2k}},\quad k\geq 1.
\end{align}
 By conjugating \eqref{uEvenL1} we note that the positive even Fourier modes also satisfy
\begin{align}\label{uEvenL1+}
\ol{\del} u_{2k+2} + \del u_{2k} = 0,\quad k\geq 0.
\end{align}Moreover, they extend continuously to $\Gam$ and
\begin{align}
u_{2k}|_\Gam=\ol{u_{-2k}}|_\Gam=\ol{g_{-2k}}=g_{2k},\quad k\geq 1.
\end{align}

Thus, as a summary, we have shown that
\begin{align}\label{PostiveEvens}
&\ol{\del} u_{2k} + \del{u_{2k-2}} = 0, \quad \forall k\in \mathbb{Z}, \\ \label{PostiveEvensTrace}
&u_{2k}\lvert_{\Gam} = g_{2k}, \quad \forall k\in \mathbb{Z}.
\end{align}

{\bf Step 3: The construction of modes $u_{-1}$ and $u_{1}$.}

Let $\psi \in \Psi_{g}$ as in \eqref{RT_PsiClass}.
We define 
\begin{align}\label{uminus1Psi}
u_{-1} := \psi, \quad \text{and} \quad u_{1} := \ol{\psi}.
\end{align}
Since $g$ is real valued, we have 
\begin{align}\label{u1trace}
u_1 \lvert_{\Gam}= \ol{g_{-1}}= g_{1}.
\end{align}

{\bf Step 4: The construction of negative odd modes $u_{2n-1}$ for $n \leq -1$.}

Use the Bukhgeim-Cauchy Integral formula \eqref{BukhgeimCauchyFormula} to 
construct the other odd negative Fourier modes:
\begin{align}\label{constructionODDSnegative}
\langle u_{-3}(z), u_{-5}(z),  \cdots \rangle
:= \B \bg^{odd}(z), \quad z\in \OM.
\end{align} 
By Theorem \ref{BukhgeimCauchyThm}, the sequence valued map
\begin{align*}
z\mapsto \langle u_{-3}(z), u_{-5}(z), u_{-7}(z),  ..., \rangle ,
\end{align*}is $\mathcal{L}$-analytic in $\OM$, thus the equations
\begin{align} \label{negODDS}
\ol{\del} u_{2k-1} + \del u_{2k-3} &= 0, 
\end{align} are satisfied for all $k\leq -1$.
Moreover, the hypothesis \eqref{RT2TensorCond2} and the sufficiency part of Theorem \ref{NecSuf} yields that they extend continuously to $\Gam$ and
\begin{align}\label{NegOddsTrace}
u_{2k-1}|_\Gam=g_{2k-1},\quad \forall k\leq -1.
\end{align}

{\bf Step 5: The construction of positive odd modes $u_{2n+1}$ for $n \geq 1$.}

All of the positive odd Fourier modes are constructed by conjugation:
\begin{align}\label{construct_odd+}
u_{2k+3}:=\ol{u_{-(2k+3)}},\quad k\geq 0.
\end{align}By conjugating \eqref{negODDS} we note that the positive odd Fourier modes also satisfy
\begin{align}\label{PostiveOdds}
\ol{\del} u_{2k+3} + \del{u_{2k+1}} = 0, \quad \forall k \geq 1.
\end{align} Moreover, they extend continuously to $\Gam$ and
\begin{align}\label{PostiveOddsTrace}
u_{2k+3}|_\Gam=\ol{u_{-(2k+3)}}|_\Gam=\ol{g_{-(2k+3)}}=g_{2k+3},\quad k\geq 0.
\end{align}

{\bf Step 6: The construction of the tensor field $\bF_{\psi} $ whose $X$-ray data is $g$.}

We define the 2-tensor field
\begin{align}\label{Fdefn_a=0}
\bF_{\psi} := \begin{pmatrix} f_{0} + 2\re f_2  & 2 \im f_2 \\ 2 \im f_2 & f_{0} - 2\re f_2 \end{pmatrix},
\end{align}
 where
\begin{align}\label{f0_f2_defn_a=0}
f_{0}=2 \re (\del\psi), \; \text{and} \; f_{2}=\ol{\del}\psi+\del u_{-3}.
\end{align}
In order to show $g \lvert_{\Gam_{+}} = X\bF_{\psi}$ with $\bF_{\psi}$ as in \eqref{Fdefn_a=0}, 
we define the real valued function $u$ via its Fourier modes
\begin{align}\label{definitionU_a=0}
u(z, \tta)&:=u_{0}(z)+ \psi(z)e^{-i \fii} + \ol{\psi}(z)e^{i \fii} \\ \nonumber
&\qquad + \sum_{n = 2}^{\INF} u_{-n}(z)e^{-i n\fii} + \sum_{n = 2}^{\INF} u_{n}(z)e^{i n\fii},
\end{align} and check that it has the trace $g$ on $\Gam$ and satisfies the transport equation \eqref{2TtransportEQ_a=0}.

Since $g\in C^{\alpha} \left(\Gam; C^{1,\alpha}(\sph) \right)\cap C(\Gam;C^{2,\alpha}(\sph))$, we use \cite[Corollary 4.1]{sadiqtamasan01} and
\cite[Proposition 4.1 (iii)]{sadiqtamasan01} to conclude that $u$ defined in \eqref{definitionU_a=0} belongs to $C^{1,\alpha}(\OM \times \sph)\cap C^{\alpha}(\ol{\OM}\times \sph)$. In particular $u(\cdot,\tta)$ for $\tta= (\cos\fii,\sin\fii)$ extends to the boundary and its trace satisfies
\begin{align*}
u(\cdot,\tta)\lvert_{\Gam} &=\left.\left ( u_{0}+ \psi e^{-i \fii} + \ol{\psi}e^{i \fii} + \sum_{n = 2}^{\INF} u_{-n}e^{-i n\fii} + \sum_{n = 2}^{\INF} u_{n}e^{i n\fii}\right) \right\lvert_{\Gam}\\
&=u_{0}\lvert_{\Gam}+ \psi \lvert_{\Gam}\,e^{-i \fii} + \ol{\psi}\lvert_{\Gam}\,e^{i \fii} + \sum_{n = 2}^{\INF} u_{-n}\lvert_{\Gam}e^{-i n\fii} + \sum_{n = 2}^{\INF} u_{n}\lvert_{\Gam}e^{i n\fii}\\
&=g_{0}+ g_{-1}e^{-i \fii} +g_{1}e^{i \fii}+\sum_{n = 2}^{\INF} g_{-n}e^{-i n\fii} + \sum_{n = 2}^{\INF} g_{n}e^{i n\fii} \\
&= g(\cdot,\tta),
\end{align*} where in the third equality above we used \eqref{PostiveEvensTrace}, \eqref{NegOddsTrace},\eqref{PostiveOddsTrace}, \eqref{u1trace} and definition of $\psi \in \Psi_{g}$ in \eqref{RT_PsiClass}.

Since $u\in C^{1,\alpha}(\OM \times \sph)\cap C^{\alpha}(\ol{\OM}\times \sph)$, the following calculation is also justified:
\begin{align*}
\tta \cdot \nabla u &= e^{-i \fii} \ol{\del}u_0 + e^{i \fii} \del u_0 +
+e^{-2i \fii} \ol{\del}\psi + \ol{\del \psi} + \del \psi +e^{2i \fii} \del \ol{\psi} \\
& \qquad + \sum_{n = 2}^{\INF} \ol{\del}u_{-n}e^{-i (n+1) \fii} + \sum_{n = 2}^{\INF} \del u_{-n}e^{-i (n-1) \fii} \\
& \qquad + \sum_{n = 2}^{\INF} \ol{\del}u_{n}e^{i (n-1) \fii} + \sum_{n = 2}^{\INF} \del u_{n}e^{i (n+1) \fii}.
\end{align*}
Rearranging the modes in the above equation yields
\begin{align*}
\tta \cdot \nabla u &=  e^{-2i \fii} (\ol{\del}\psi+\del u_{-3})  + e^{2i \fii} (\del \ol{\psi} + \ol{\del}u_{3}) + \ol{\del \psi} + \del \psi \\
& \qquad + e^{-i \fii} (\ol{\del}u_0 +\del u_{-2}) + e^{i \fii} (\del u_0 +\ol{\del}u_2)  \\
& \qquad +\sum_{n = 1}^{\INF} (\ol{\del}u_{-n}+ \del u_{-n-2})e^{-i (n+1) \fii} + \sum_{n = 1}^{\INF} ( \ol{\del}u_{n+2}+\del u_{n}) e^{i (n+1) \fii}.
\end{align*}
Using \eqref{PostiveEvens}, \eqref{negODDS}, and \eqref{PostiveOdds} simplifies the above equation
\begin{align*}
\tta \cdot \nabla u &= e^{-2i \fii} (\ol{\del}\psi+\del u_{-3})  + e^{2i \fii} (\del \ol{\psi} + \ol{\del}u_{3}) + \ol{\del \psi} + \del \psi.
\end{align*}
Now using \eqref{f0_f2_defn_a=0}, we conclude \eqref{2TtransportEQ_a=0}.
\begin{align*}
\tta \cdot \nabla u &= e^{-2i \fii} f_2 + e^{2i \fii} \ol{f_2}+f_0 = \langle \bF_{\psi} \tta ,  \tta \rangle.
\end{align*}

\end{proof}
As the source is supported inside, there are no incoming fluxes: hence the trace
of a solution $u$ of  \eqref{2TtransportEQ_a=0} on $\Gam_{-}$ is zero. We give next a range condition only in terms of $g$ on $\Gam_{+}$, where $g:= u \lvert_{\Gam \times \sph}$.
More precisely, let $\tilde{u}$ be the solution of the boundary value problem
\begin{align}\label{Ref_TtransportModel}
\begin{aligned}
\tta\cdot\nabla \tilde{u}(x,\tta)&= \langle \bF(x)\tta,  \tta \rangle, \quad x\in\OM, \\
 \tilde{u}(z, \tta)&= -\frac{1}{2}g \lvert_{\Gam_{+}}(z,-\tta), \quad (z,\tta)\in \Gam_{-}.
\end{aligned}
\end{align}
Then one can see that 
\begin{align}\label{utilde+}
\tilde{u} \lvert_{\Gam_{+}} = \frac{1}{2} g \lvert_{\Gam_{+}},
\end{align}
and therefore $\tilde{u} \lvert_{\Gam \times \sph}$ is an odd function of $\tta$.
This shows that we can work with the following odd extension:
\begin{align}\label{gtildedefn}
\tilde{g}(z,\tta) := \frac{g(z,\tta)-g(z,-\tta)}{2}, \quad (z,\tta) \in (\Gam \times \sph) \backslash \Gam_{0,}
\end{align}
and $\tilde{g}=0 $ on $\Gam_0$. 
Note that  $\tilde{g}$ is the trace of $\tilde{u}$ on $\Gam \times \sph$.

The range characterization can be given now in terms of the  odd Fourier modes of $\tilde{g}$, namely in terms of
\begin{align}\label{gtildeFourier}
 &\tilde{\bg}:=\langle \tilde{g}_{-3}, \tilde{g}_{-5}, \tilde{g}_{-7}, ... \rangle.
 \end{align}

\begin{cor}\label{RefineRT2Tensor}

Let $\alpha > 1/2$.

(i) Let $\bF \in C_{0}^{1,\alpha}(\OM;\BR^{2 \times 2})$, $\tilde{u}$ be the solution of \eqref{Ref_TtransportModel} and $\tilde{\bg}$ as in \eqref{gtildeFourier}.
Then $\tilde{\bg} \in l^{1,1}_{\INF}(\Gamma)\cap C^\alpha(\Gamma;l_1)$ and
\begin{align}\label{Ref_RT2TensorCond1}
&[I+i\HT] \tilde{\bg} =0,
\end{align} where the operator $\HT$ is the Hilbert transform in \eqref{hilbertT}.

(ii) Let $g\in C^{\alpha} \left(\Gam; C^{1,\alpha}(\sph) \right)\cap C(\Gam;C^{2,\alpha}(\sph))$ be 
real valued  with $g \lvert_{\Gam_{-} \cup \Gam_{0}} =0$.
Let $\tilde{g}$ be its odd extension as in \eqref{gtildedefn} and the corresponding $\tilde{\bg}$ as in \eqref{gtildeFourier}. If $\tilde{\bg}$  satisfies \eqref{Ref_RT2TensorCond1},
then there exists a real valued symmetric 2-tensor $\bF \in C(\OM;\BR^{2 \times 2})$, such that $g \lvert_{\Gam_{+}}= X\bF$.
Moreover for each $\psi \in \Psi_{g}$ in \eqref{RT_PsiClass}, there is a unique real valued symmetric 2-tensor $\bF_{\psi}$ such that
$g\lvert_{\Gam_{+}} = X\bF_{\psi}$.
\end{cor}

\section{The attenuated case} \label{Range2TensorAtten}

In this section we assume an attenuation $a \in C^{2,\alpha}(\ol{\OM})$, $\alpha>1/2$ with 
\begin{align*}
\underset{\ol{\OM}}{\min}\, a > 0.
\end{align*}
We establish necessary and sufficient conditions for a sufficiently smooth 
function $g$ on $\Gam \times \sph$ to be the attenuated $X$-ray data, with attenuation $a$, of some sufficiently smooth real symmetric 2-tensor, i.e. $g$ is the trace on $\Gam\times\sph$ of some solution $u$ of
\begin{align}\label{TransportEq2Tensor2}
\tta\cdot\nabla u(x,\tta) +a(x) u(x,\theta) &= \langle \bF(x)\tta, \tta \rangle, \quad (x,\tta) \in \Gam \times \sph.
\end{align}

Different from 1-tensor case in \cite{sadiqtamasan02} (where there is uniqueness), in the 2-tensor case there is non-uniqueness: see the class of function in \eqref{PsiClass}.

As in \cite{sadiqtamasan01} we start by the reduction to the non-attenuated case via the special integrating factor $e^{-h}$, where $h$ is explicitly defined in terms of $a$ by
\begin{align}\label{hDefn}
h(z,\tta) := Da(z,\theta) -\frac{1}{2} \left( I - i H \right) Ra(z\cdot \tta^{\perp},\tta),
\end{align}where $\tta^\perp$ is  orthogonal  to $\tta$, 
$Da(z,\tta) =\ds \int_{0}^{\INF} a(z+t\tta)dt$ is the divergence beam transform of the attenuation $a$, 
$Ra(s,\tta) = \ds \int_{-\INF}^{\INF} a\left( s \tta^{\perp} +t \tta \right)dt$ is the Radon transform of
the attenuation $a$, and
the classical Hilbert transform $H h(s) = \ds \frac{1}{\pi} \int_{-\INF}^{\INF} \frac{h(t)}{s-t}dt $ is 
taken in the first variable and evaluated at $s = z \cdotp \tta^{\perp}$.  The function $h$  was first considered in the work of Natterer \cite{naterrerBook}; see also \cite{finch}, and \cite{bomanStromberg} for elegant arguments that show how $h$ extends from $\sph$ inside the disk as an analytic map.

The lemma \ref{hproperties} and lemma \ref{intrinsic} below were proven in \cite{sadiqtamasan02} for $a$ vanishing at the boundary, $a\in C^{1,\alpha}_{0}(\ol \OM)$, $\alpha>1/2$. We explain here why the vanishing assumption is not necessary: we extend $a$ in a neighbourhood $\tilde{\OM}$ of $\OM$ with compact support, $\tilde{a} \in C_{0}^{1,\alpha}(\tilde{\OM})$. We apply the results 
\cite[Lemma 4.1 and Lemma 4.2]{sadiqtamasan02} for the extension $\tilde{a}$ and use it on $\ol{\OM}$.

\begin{lemma}\cite[Lemma 4.1]{sadiqtamasan02}\label{hproperties}
Assume $a\in C^{p,\alpha}(\ol \OM)$, $p = 1,2$, $\alpha>1/2$, and $h$ defined in \eqref{hDefn}. Then $h \in C^{p,\alpha}(\ol \OM \times \sph)$ and the following hold

(i) $h$ satisfies \begin{align}\label{h=IntegratingFactor}
\tta \cdot \nabla h(z,\tta) = -a(z), \; (z, \tta) \in \OM \times \sph.
\end{align}

(ii) $h$ has vanishing negative Fourier modes yielding the expansions
\begin{align}\label{ehEq}
  e^{- h(z,\tta)} := \sum_{k=0}^{\INF} \alpha_{k}(z) e^{ik\fii}, \quad e^{h(z,\tta)} := \sum_{k=0}^{\INF} \beta_{k}(z) e^{ik\fii}, \; (z, \tta) \in \ol\OM \times \sph,
\end{align}
with

(iii)
\begin{align}
&z\mapsto \langle \alpha_{1}(z), \alpha_{2}(z), \alpha_{3}(z), ... , \rangle \in C^{p,\alpha}(\OM ; l_{1})\cap C(\ol\OM ; l_{1}), \\
&z\mapsto \langle \beta_{1}(z), \beta_{2}(z), \beta_{3}(z), ... , \rangle \in C^{p,\alpha}(\OM ; l_{1})\cap C(\ol\OM ; l_{1}).\label{betasDecay}
\end{align}

(iv) For any $z \in \OM $
\begin{align}\label{betazero}
&\ol{\del} \beta_0(z) = 0, \\ \label{betaone}
& \ol{\del} \beta_1(z) = -a(z) \beta_0(z),\\ \label{betak}
& \ol{\del} \beta_{k+2}(z) +\del \beta_{k}(z) +a(z) \beta_{k+1}(z)=0, \; k \geq 0.
\end{align}

(v) For any $z \in \OM $
\begin{align}\label{alphazero}
&\ol{\del} \alpha_0(z) = 0, \\ \label{alphaone}
& \ol{\del} \alpha_1(z) = a(z) \alpha_0(z),\\ \label{alphak}
& \ol{\del} \alpha_{k+2}(z) +\del \alpha_{k}(z) +a(z) \alpha_{k+1}(z)=0, \; k \geq 0.
\end{align}

(vi) The Fourier modes $\alpha_{k}, \beta_{k}, k\geq 0$ satisfy
\begin{align}\label{alphabetaSys}
\alpha_0 \beta_0 =1, \quad \sum_{m=0}^{k} \alpha_{m}\beta_{k-m}=0, k \geq 1.
\end{align}
\end{lemma}

From \eqref{h=IntegratingFactor} it is easy to see that $u$ solves \eqref{TransportEq2Tensor2} if and only if $v:=e^{-h}u$ solves
\begin{align}\label{transp_a>0_in_v}
\tta\cdot\nabla v(z,\tta)=\langle F(z) \tta, \tta \rangle e^{-h(z,\tta)}.
\end{align}

If $ u(z,\tta) = \sum_{n =-\INF}^{\INF} u_{n}(z) e^{in\fii} $ solves  \eqref{TransportEq2Tensor2}, then its Fourier modes satisfy
\begin{align} \label{AfzeroEq}
&\ol{\del} u_{1}(z) + \del u_{-1}(z) +a(z)u_{0}(z) = f_{0}(z), \\  \label{Auminus1Eq}
&\ol{\del} u_{0}(z) + \del u_{-2}(z) +a(z)u_{-1}(z) = 0, \\ \label{Af2Eq}
&\ol{\del} u_{-1}(z) + \del u_{-3}(z) +a(z)u_{-2}(z) = f_{2}(z), \\ \label{AttenRT2AnalyticEq}
&\ol{\del} u_{n}(z) + \del u_{n-2}(z) +a(z)u_{n-1}(z) = 0, \quad n \leq -2,
\end{align} where $f_{0}, f_{2}$ as defined in \eqref{f0f2Eq}.

Also, if $v:=e^{-h}u=\sum_{n =-\INF}^{\INF} v_{n}(z) e^{in\fii}$ solves \eqref{transp_a>0_in_v}, then its Fourier modes satisfy
\begin{align}
&\ol{\del} v_{1}(z) + \del v_{-1}(z) = \alpha_0(z){f_0}(z) + \alpha_2(z){f_{2}}(z),\nonumber \\
&\ol{\del} v_{0}(z) + \del v_{-2}(z)= \alpha_1(z){f_2}(z),  \nonumber \\
&\ol{\del} v_{-1}(z) + \del v_{-3}(z)= \alpha_0(z){f_2}(z),  \nonumber \\
&\ol{\del} v_{n}(z) + \del v_{n-2}(z)= 0, \quad n \leq -2,\label{AttenRT2AnalyticEq_v}
\end{align}
where $\alpha_0$, $\alpha_1$ and $\alpha_2$ are the Fourier modes in \eqref{ehEq}, and $f_{0}, f_{2}$ as defined in \eqref{f0f2Eq}.

The following result shows that the equivalence between \eqref{AttenRT2AnalyticEq} and \eqref{AttenRT2AnalyticEq_v} is intrinsic to negative Fourier modes only.
\begin{lemma}\cite[Lemma 4.2]{sadiqtamasan02}\label{intrinsic} Assume $a\in C^{1,\alpha}(\ol \OM), \alpha>1/2$.

(i) Let $\bv= \langle v_{-2}, v_{-3}, ... \rangle \in C^1(\OM, l_1)$ satisfy \eqref{AttenRT2AnalyticEq_v}, and $\bu = \langle u_{-2}, u_{-3}, ... \rangle$ be defined componentwise by the convolution
  \begin{align}\label{UintermsV}
 u_{n} := \sum_{j=0}^{\INF} \beta_j v_{n-j}, \quad n \leq -2,
 \end{align} where $\beta_j$'s are the Fourier modes in \eqref{ehEq}.
 Then $\bu$ solves \eqref{AttenRT2AnalyticEq} in $\OM$.

 (ii) Conversely, let $\bu= \langle u_{-2}, u_{-3}, ... \rangle \in C^1(\OM, l_1)$ satisfy \eqref{AttenRT2AnalyticEq}, and $\bv = \langle v_{-2}, v_{-3}, ... \rangle$ be defined componentwise by the convolution
  \begin{align}\label{VintermsU}
 v_{n} := \sum_{j=0}^{\INF} \alpha_j u_{n-j}, \quad n \leq -2,
 \end{align} where $\alpha_j$'s are the Fourier modes in \eqref{ehEq}.
 Then $\bv$ solves \eqref{AttenRT2AnalyticEq_v} in $\OM$.
\end{lemma}

The operators $\del,\, \ol{\del}$ in \eqref{CauchyRiemannOp} can be rewritten in terms of the derivative in tangential direction  $\del_{\tau}$ and derivative in normal direction $\del_n$,
\begin{align*}
&\del_n = \cos \eta \del_{x_1} +\sin \eta \del_{x_2}, \\
&\del_\tau = -\sin \eta \del_{x_1} +\cos \eta \del_{x_2},
\end{align*} where $\eta$ is the angle made by the normal to the boundary with $x_1$ direction (Since the boundary $\Gam$ is known, $\eta$ is a known function on the boundary).
In these coordinates
\begin{align}\label{del_interms_delR_ETA}
\del = \frac{e^{-i \eta}}{2} (\del_n -i \del_\tau), \quad \ol{\del} = \frac{e^{i \eta}}{2} (\del_n +i \del_\tau).
\end{align}

Next we characterize the attenuated $X$-ray data $g$ in terms of 
its Fourier modes $g_{0},  g_{-1}$ and the negative index modes $\gamma_{-2}, \gamma_{-3}, \gamma_{-4} ...$ of
\begin{align}\label{FourierData_a>0}
e^{- h(\zeta,\tta)} g(\zeta,\theta) = \sum_{k=-\INF}^{\INF} \gamma_{k}(\zeta) e^{ik\fii}, \quad \zeta \in \Gam.
\end{align}
To simplify the statement, let
\begin{align}\label{g_h}
\bg_h:=\langle \gamma_{-2}, \gamma_{-3}, \gamma_{-4} ... \rangle,
\end{align}
and from the negative even, respectively, negative odd Fourier modes, we built the sequences
\begin{align}\label{gHEvenOdd}
\bg_h^{even} = \langle \gamma_{-2},\gamma_{-4},...\rangle, \quad \text{and} \quad
\bg_h^{odd} = \langle \gamma_{-3},\gamma_{-5},...\rangle.
\end{align} Note that $\gamma_{-1}$ is not included  in the $\bg_h^{odd}$ definition.
As before we construct simultaneously the right hand side of the transport equation \eqref{TransportEq2Tensor2} together with the solution $u$. Construction of $u$ is via its Fourier  modes. We first construct the negative modes and then the positive modes are constructed by conjugation.  Apart from zeroth mode $u_0$ and negative one mode $u_{-1}$, all Fourier modes are constructed uniquely from the data $\bg_h^{even}$, $\bg_h^{odd}$.
The mode $u_{0}$ will be chosen arbitrarily from the class $\Psi_{g}^{a}$ with prescribed trace and gradient on the boundary $\Gam$ defined as
\begin{align}\label{PsiClass}
\Psi_{g}^{a}:=
&\left \{ \vphantom{\int} 
\psi \in C^{2}(\ol\OM ;\BR):\psi \lvert_{\Gam}= g_{0},
\right.  \\ \nonumber
&\left. \vphantom{\int}
  \del_n \psi  \lvert_{\Gam}= -2 \re  e^{-i\eta} \left ( \left. \del \sum_{j=0}^{\INF} \beta_j  (\B \bg_h )_{-2-j} \right \lvert_{\Gam} +  a \lvert_{\Gam} \,  g_{-1} \right )  \right \},
\end{align}
 where $\B$ be the Bukhgeim-Cauchy operator in  \eqref{BukhgeimCauchyFormula}, $\beta_{j}$'s are the Fourier modes in \eqref{ehEq} and $\bg_h$ in \eqref{g_h}.
The mode $u_{-1}$ is define in terms of $u_0$, see \eqref{DefnUminus1}. 

Recall the Hilbert transform $\HT$ in \eqref{hilbertT}.
\begin{theorem}[Range characterization in the attenuated case]\label{ART2Tensor}

Let $a \in C^{2,\alpha}(\ol\OM)$, $\alpha>1/2$ with $\underset{\ol{\OM}}{\min}\, a >0$.

(i) Let $\bF \in C_{0}^{1,\alpha}(\OM;\BR^{2 \times 2})$. For $g := \left \{ \begin{array}{ll}
X_{a} \bF(x,\tta) , & (x,\tta) \in \Gam_{+}, \\
0 , &  (x,\tta) \in \Gam_{-} \cup \Gam_{0}, \\
\end{array}
  \right. $ consider the corresponding sequences $\bg_h^{even}, \bg_h^{odd}$ as in \eqref{gHEvenOdd}.
 Then $\bg_h^{even}, \bg_h^{odd} \in l^{1,1}_{\INF}(\Gamma)\cap C^\alpha(\Gamma;l_1)$ satisfy
\begin{align}\label{ART2TensorCond1}
&[I+i\HT] \bg_h^{even} =0, \quad [I+i\HT] \bg_h^{odd} =0, \quad \text{and} \\ \label{ART2TensorCond2}
&\del_{\tau} g_0  = -2 \im  e^{-i\eta} \left ( \left. \del \sum_{j=0}^{\INF} \beta_j  (\B \bg_h )_{-2-j} \right \lvert_{\Gam} +  a \lvert_{\Gam} \,  g_{-1} \right ) ,
\end{align} where $\HT$ is the Hilbert transform in \eqref{hilbertT}, $\B$ is the Bukhgeim-Cauchy operator in \eqref{BukhgeimCauchyFormula}, $\beta_{j}$'s are the Fourier modes in \eqref{ehEq} and $\bg_h$ in \eqref{g_h}.

(ii) Let $g\in C^{\alpha} \left(\Gam; C^{1,\alpha}(\sph) \right)\cap C(\Gam;C^{2,\alpha}(\sph))$ be real valued with $g \lvert_{\Gam_{-} \cup \Gam_{0}}=0$.
If the corresponding sequences $\bg_h^{even}, \bg_h^{odd}\in Y_{\alpha}$ satisfying \eqref{ART2TensorCond1} and \eqref{ART2TensorCond2}  then there exists a symmetric 2-tensor  $\bF \in C(\OM;\BR^{2 \times 2})$,
such that $g \lvert_{\Gam_{+}}= X_{a}\bF$. Moreover for each $\psi \in \Psi_{g}^{a}$ in \eqref{PsiClass}, there is a unique real valued symmetric 2-tensor $\bF_{\psi}$ such that
$g\lvert_{\Gam_{+}} = X_{a}\bF_{\psi}$.
\end{theorem}

\begin{proof}
(i) {\bf Necessity}

Let $\bF \in C_{0}^{1,\alpha}(\OM;\BR^{2 \times 2})$. Since $\bF$ is compactly supported inside \
$\OM$, for any point at the boundary there is a cone of lines which do not meet the support.
Thus $g \equiv 0$ in the neighborhood of the variety $\Gam_0$
 which yields $g \in C^{1,\alpha}(\Gam \times \sph)$. Moreover, $g$ is the trace on $\Gam \times \sph$ of a solution $u \in C^{1,\alpha}(\ol{\OM} \times \sph)$.
By \cite[Proposition 4.1]{sadiqtamasan01} $\bg_{h}^{even}, \bg_{h}^{odd} \in l^{1,1}_{\INF}(\Gamma)\cap C^\alpha(\Gamma;l_1)$. 

Let $v:=e^{-h}u=\sum_{n =-\INF}^{\INF} v_{n}(z) e^{in\fii}$, then the negative Fourier modes of $v$ satisfy \eqref{AttenRT2AnalyticEq_v}. In particular its negative odd subsequence $\langle v_{-3},v_{-5},...\rangle$ and
negative even subsequence $ \langle v_{-2},v_{-4},...\rangle$ are $\mathcal{L}$-analytic with traces $\bg_{h}^{odd}$ respectively $\bg_{h}^{even}$.
The necessity part of Theorem \ref{NecSuf} yields \eqref{ART2TensorCond1}:
\begin{align*}
[I+i\HT]\bg_{h}^{odd}=0,\quad [I+i\HT]\bg_{h}^{even}=0.
\end{align*}

If $u$ solves \eqref{TransportEq2Tensor2}, then its Fourier modes satisfy \eqref{AfzeroEq}, \eqref{Auminus1Eq}, \eqref{Af2Eq}, and \eqref{AttenRT2AnalyticEq}. 
The negative Fourier modes of $u$ and $v$ are related by 
\begin{align}\label{Un_Vn}
u_{n} = \sum_{j=0}^{\INF} \beta_j v_{n-j}, \quad n \leq 0,
\end{align} where $\beta_{j}$'s are the Fourier modes in \eqref{ehEq}.
The restriction of \eqref{Auminus1Eq} to the boundary yields
\begin{align*}
\ol{\del} u_{0} \lvert_{\Gam} &= - \del u_{-2} \lvert_{\Gam} - (a u_{-1}) \lvert_{\Gam}.
\end{align*}
Expressing $\ol{\del}$ in the above equation in terms of $\del_{\tau}$ and $\del_n$ as in \eqref{del_interms_delR_ETA} yields  
\begin{align*}
\frac{e^{i \eta}}{2} (\del_n +i \del_\tau) u_{0} \lvert_{\Gam}  &= -\del u_{-2} \lvert_{\Gam} - a\lvert_{\Gam} g_{-1}.
\end{align*}
Simplifying the above expression and using $\del_\tau u_{0} \lvert_{\Gam} = \del_\tau g_{0}$, yields
\begin{align*}
 \del_n u_{0} \lvert_{\Gam} +i \del_\tau g_{0} &= -2e^{-i \eta}\left ( \del u_{-2} \lvert_{\Gam} + a \lvert_{\Gam} g_{-1} \right ).
\end{align*}
The imaginary part of the above equation yields \eqref{ART2TensorCond2}. This proves part (i) of the theorem.

(ii) {\bf Sufficiency}

To prove the sufficiency we will construct a real valued symmetric 2-tensor $\bF$ in $\OM$ and a real valued function $u \in C^{1}(\OM\times \sph)\cap C(\ol\OM \times \sph)$
such that $u \lvert_{\Gam \times \sph}=g$ and $u$ solves \eqref{TransportEq2Tensor2} in $\OM$.
The construction of such $u$ is in terms of its Fourier modes in the angular variable and it is done in several steps.

{\bf Step 1: The construction of negative modes $u_{n}$ for $n \leq -2$.}

Let $g\in C^{\alpha} \left(\Gam; C^{1,\alpha}(\sph) \right)\cap C(\Gam;C^{2,\alpha}(\sph))$ be real valued
with $g \lvert_{\Gam_{-} \cup \Gam_{0}}=0$.
Let the corresponding sequences $\bg_h^{even}, \bg_h^{odd}$ as in \eqref{gHEvenOdd} satisfying \eqref{ART2TensorCond1} and \eqref{ART2TensorCond2}.
By \cite[Proposition 4.1(ii)]{sadiqtamasan01} and \cite[Proposition 5.2(iii)]{sadiqtamasan01} $\bg_h^{even}, \bg_h^{odd}\in Y_{\alpha}$.
Use the Bukhgeim-Cauchy Integral formula \eqref{BukhgeimCauchyFormula} to define the $\mathcal{L}$-analytic maps
\begin{align}
&\bv^{even}(z)= \langle v_{-2}(z), v_{-4}(z), ... \rangle:= \B \bg_h^{even}(z), \quad z\in \OM,\\
&\bv^{odd}(z)= \langle v_{-3}(z), v_{-5}(z), ... \rangle:= \B \bg_h^{odd}(z),\quad z\in \OM.
\end{align}By intertwining let also define
\begin{align*}
\bv(z):=\langle v_{-2}(z),v_{-3}(z), ... \rangle , \quad z\in \OM.
\end{align*}
By Proposition \ref{extraRegularity}
\begin{align}\label{smothness__v_j}
\bv^{even}, \bv^{odd},\bv\in C^{1,\alpha}(\OM; l_{1})\cap C^{\alpha}(\ol\OM;l_1)\cap C^2(\OM;l_\infty).
 \end{align}Moreover, since $\bg_h^{even}, \bg_h^{odd}$ satisfy the hypothesis \eqref{ART2TensorCond1}, by Theorem \ref{NecSuf} we have
\begin{align*}
\bv^{even} \lvert_{\Gam} = \bg_h^{even} \quad \text{and}\quad  \bv^{odd}\lvert_{\Gam} = \bg_h^{odd}.
\end{align*} In particular
\begin{align}\label{vn_intermsof_gn}
v_{n} \lvert_{\Gam} = \sum_{k=0}^{\INF} \left(\alpha_{k}\lvert_\Gam\right) g_{n-k}, \quad n \leq -2.
\end{align}

For each $n \leq -2$, we use the convolution formula below to construct
 \begin{align}\label{DefnUnConv}
 u_{n} := \sum_{j=0}^{\INF} \beta_j v_{n-j}.
 \end{align}
Since $a\in C^{2,\alpha}(\ol\OM)$, by \eqref{betasDecay}, the sequence
 $z \mapsto \langle \beta_0(z), \beta_1(z), \beta_2(z),...\rangle$ is in $C^{2,\alpha}(\OM; l_{1}) \cap C^{\alpha}(\ol\OM ; l_1)$. Since convolution preserves $l_1$, the map is in
 \begin{align}\label{regularityofU}
 z\mapsto \langle u_{-2}(z), u_{-3}(z),...\rangle  \in C^{1,\alpha}(\OM; l_{1}) \cap C^{\alpha}(\ol\OM ; l_1).
 \end{align} Moreover, since $\bv\in C^2(\OM;l_\infty)$ as in \eqref{smothness__v_j}, we also conclude from convolution that
 \begin{align}\label{u_minus2_reg}
 z\mapsto \langle  u_{-2}(z), u_{-3}(z),...\rangle  \in C^{2}(\OM; l_\infty).
 \end{align}
The property \eqref{regularityofU} justifies the calculation of traces $u_n|_\Gam$
for each $n \leq -2$:
\begin{align*}
u_{n} \lvert_{\Gam} &= \sum_{j=0}^{\INF} \beta_j \lvert_{\Gam} \left( v_{n-j} \lvert_{\Gam} \right).
\end{align*}
Using \eqref{vn_intermsof_gn} in the above equation gives
\begin{align*}
u_{n} \lvert_{\Gam} &=\sum_{j=0}^{\INF} \beta_j \lvert_{\Gam} \sum_{k=0}^{\INF} \alpha_{k} \lvert_{\Gam} g_{n-j-k}.
\end{align*}
A change of index $m=j+k$, simplifies the above equation
\begin{align*}
u_{n} \lvert_{\Gam} &= \sum_{m=0}^{\INF} \sum_{k=0}^{m}  \alpha_{k} \beta_{m-k} g_{n-m}, \\
&= \alpha_{0}\beta_{0}g_{n} +\sum_{m=1}^{\INF} \sum_{k=0}^{m}  \alpha_{k} \beta_{m-k} g_{n-m}.
\end{align*}
Using Lemma \ref{hproperties} (vi) yields
\begin{align}\label{Trace_un_neg_gn}
u_{n} \lvert_{\Gam} &= g_{n}, \quad n \leq -2.
\end{align}

From the Lemma \ref{intrinsic}, the constructed $u_{n}$ in \eqref{DefnUnConv} satisfy
\begin{align}\label{uL2sys_a>0}
\ol{\del} u_{n} + \del u_{n-2} +a u_{n-1} = 0, \quad n\leq-2.
\end{align}

{\bf Step 2: The construction of positive modes $u_{n}$ for $n \geq 2$.}

All of the positive Fourier modes are constructed by conjugation:
\begin{align}\label{construct_pos_Fmodes}
u_{n}:=\ol{u_{-n}},\quad n\geq 2.
\end{align} Moreover using \eqref{Trace_un_neg_gn}, the traces $u_n|_\Gam$ for each $n \geq 2$:
\begin{align}\label{Trace_un_pos_gn}
u_{n}\lvert_{\Gam} = \ol{u_{-n}}\lvert_{\Gam}= \ol{g_{-n}}= g_{n}, \quad n \geq 2.
\end{align}

By conjugating \eqref{uL2sys_a>0} we note that the positive Fourier modes also satisfy
\begin{align}\label{uL2sys_a>0_Pos}
\ol{\del} u_{n+2} + \del u_{n} +a u_{n+1} = 0, \quad n\geq2.
\end{align}

{\bf Step 3: The construction of modes $u_{0}, u_{-1}$ and $u_{1}$.}

Let $\psi \in \Psi_{g}^{a}$ as in \eqref{PsiClass} and define
\begin{align}\label{DefnU0}
u_{0}:= \psi,
\end{align} and 
\begin{align}\label{DefnUminus1}
u_{-1}:= \frac{-\ol{\del}\psi-\del u_{-2}}{a}, \quad u_{1} := \ol{u_{-1}}.
\end{align}
By the construction $u_0 \in C^{2}(\OM; l_{\INF})$ and $u_{-1}\in C^{1}(\OM; l_{\INF})$, and
\begin{align}\label{d_u0_uminus1eq}
\ol{\del}u_0 +\del u_{-2}+au_{-1} =0
\end{align} is satisfied. Furthermore, by conjugating \eqref{d_u0_uminus1eq} yields
\begin{align}\label{d_u0_u1eq}
\del u_0 +\ol{\del} u_{2}+au_{1} =0.
\end{align}

Since $\psi \in \Psi_{g}^{a}$, the trace of $u_{0}$ satisfies
\begin{align}\label{TraceU0}
u_{0} \lvert_{\Gam} = g_0.
\end{align} 
We check next that the trace of $u_{-1}$ is $g_{-1}$:
\begin{align}\nonumber
u_{-1} \lvert_{\Gam} &= \left. \frac{-\ol{\del}\psi-\del u_{-2}}{a} \right \lvert_{\Gam} \\ \nonumber
 &=  - \left. \frac{1}{a} \right \lvert_{\Gam} \frac{e^{i \eta}}{2} (\del_n +i \del_\tau) \psi \lvert_{\Gam}  - \left. \frac{1}{a} \right \lvert_{\Gam}
\del u_{-2} \lvert_{\Gam}  \\ \nonumber
&= -\left. \frac{1}{2a} \right \lvert_{\Gam} e^{i \eta}
\left \{ 
\del_n \psi \lvert_{\Gam} +i \del_\tau \psi \lvert_{\Gam} + 2e^{-i \eta} \del u_{-2} \lvert_{\Gam} 
\right \} \\ \label{TraceUminus1}
&= g_{-1},
\end{align}
where the last equality uses  \eqref{ART2TensorCond2} and the condition in class \eqref{PsiClass}.

{\bf Step 4: The construction of the tensor field $\bF_{\psi} $ whose attenuated $X$-ray data is $g$.}

We define the 2-tensor
\begin{align}\label{DefnFpsi}
\bF_{\psi} := \begin{pmatrix} f_{0} + 2\re f_2  & 2 \im f_2 \\ 2 \im f_2 & f_{0} - 2\re f_2 \end{pmatrix},
\end{align} where
\begin{align}\label{f0_defn_a>0}
f_{0} &= -2 \re \left( \frac{\ol{\del}\psi+\del u_{-2}}{a} \right) +a \psi, \; \text{and} \\ \label{f2_defn_a>0}
f_{2} &= -\ol{\del} \left( \frac{\ol{\del}\psi+\del u_{-2}}{a} \right)+\del u_{-3} + a u_{-2}.
\end{align}
Note that $f_2$ is well defined as $u_{-2}\in C^{2}(\OM;l_{\INF})$ from \eqref{u_minus2_reg}.

In order to show $g \lvert_{\Gam_{+}} = X_{a} \bF_{\psi}$ with $\bF_{\psi}$ as in \eqref{DefnFpsi}, we define the real valued function $u$ via its Fourier modes
\begin{align}\label{definitionU_a>0}
u(z, \tta)&:=u_{0}(z)+ u_{-1}e^{-i \fii} + \ol{u_{-1}}(z)e^{i \fii} \\ \nonumber
&\qquad + \sum_{n = 2}^{\INF} u_{-n}(z)e^{-i n\fii} + \sum_{n = 2}^{\INF} u_{n}(z)e^{i n\fii}.
\end{align} We check below that $u$ is well defined, has the trace $g$ on $\Gam$ and satisfies the transport equation \eqref{TransportEq2Tensor2}.

For convenience consider the intertwining sequence
\begin{align*}
\bu(z):=\langle u_{0}(z),u_{-1}(z),u_{-2}(z),u_{-3}(z), ... \rangle , \quad z\in \OM.
\end{align*} Since $\bu \in C^{1,\alpha}(\OM;l_{1}) \cap C^{\alpha}(\ol{\OM};l_{1})$, by \cite[Proposition 4.1 (iii)]{sadiqtamasan01} we conclude that $u$ is well defined by \eqref{definitionU_a>0} and as a function in $C^{1,\alpha}(\OM \times \sph)\cap C^{\alpha}(\ol{\OM}\times \sph)$. In particular $u(\cdot,\tta)$ for $\tta= (\cos\fii,\sin\fii )$ extends to the boundary and its trace satisfies
\begin{align*}
u(\cdot,\tta)\lvert_{\Gam} &=\left.\left ( u_{0}+ u_{-1} e^{-i \fii} + \ol{u_{-1}}e^{i \fii} + \sum_{n = 2}^{\INF} u_{-n}e^{-i n\fii} + \sum_{n = 2}^{\INF} u_{n}e^{i n\fii}\right) \right\lvert_{\Gam}\\
&=u_{0}\lvert_{\Gam}+ u_{-1}\lvert_{\Gam}e^{-i \fii} + \ol{u_{-1}}\lvert_{\Gam}e^{i \fii} + \sum_{n = 2}^{\INF} (u_{-n}\lvert_{\Gam}) e^{-i n\fii} + \sum_{n = 2}^{\INF} (u_{n}\lvert_{\Gam})e^{i n\fii}\\
&=g_{0}+ g_{-1}e^{-i \fii} +g_{1}e^{i \fii}+\sum_{n = 2}^{\INF} g_{-n}e^{-i n\fii} + \sum_{n = 2}^{\INF} g_{n}e^{i n\fii} \\
&= g(\cdot,\tta),
\end{align*} where is the third equality we have used
\eqref{Trace_un_neg_gn}, \eqref{Trace_un_pos_gn}, \eqref{TraceU0}, and \eqref{TraceUminus1}.

Since $u\in C^{1,\alpha}(\OM \times \sph)\cap C^{\alpha}(\ol{\OM}\times \sph)$, the following calculation is also justified: 
\begin{align*}
\tta \cdot \nabla u +au &= e^{-i \fii} \ol{\del}u_0 + e^{i \fii} \del u_0
+e^{-2i \fii} \ol{\del}u_{-1} + \ol{\del} u_{1} + \del u_{-1} +e^{2i \fii} \del u_{1}  \\
& \qquad + \sum_{n = 2}^{\INF} \ol{\del}u_{-n}e^{-i (n+1) \fii} + \sum_{n = 2}^{\INF} \del u_{-n}e^{-i (n-1) \fii} \\
& \qquad + \sum_{n = 2}^{\INF} \ol{\del}u_{n}e^{i (n-1) \fii} + \sum_{n = 2}^{\INF} \del u_{n}e^{i (n+1) \fii} \\
& \qquad + au_{0}+ au_{-1} e^{-i \fii} +a u_{1}e^{i \fii} + \sum_{n = 2}^{\INF} au_{-n}e^{-i n\fii} + \sum_{n = 2}^{\INF} au_{n}e^{i n\fii}.
\end{align*}
Rearranging the modes in the above equation yields
\begin{align*}
\tta \cdot \nabla u +au&= e^{-2i \fii} (\ol{\del}u_{-1}+\del u_{-3}+au_{-2})  + e^{2i \fii} (\del u_{1} + \ol{\del}u_{3}+au_{2}) \\
& \qquad + e^{-i \fii} (\ol{\del}u_0 +\del u_{-2}+au_{-1}) + e^{i \fii} (\del u_0 +\ol{\del}u_2 +a u_{1})  \\
& \qquad + \ol{\del } u_{1}+ \del u_{-1} +au_0
+ \sum_{n = 2}^{\INF} ( \ol{\del}u_{n+2}+\del u_{n}+au_{n+1}) e^{i (n+1) \fii} \\
& \qquad
+\sum_{n = 2}^{\INF} (\ol{\del}u_{-n}+ \del u_{-n-2}+au_{-n-1})e^{-i (n+1) \fii} .
\end{align*}
Using \eqref{uL2sys_a>0}, \eqref{uL2sys_a>0_Pos}, \eqref{d_u0_uminus1eq} and \eqref{d_u0_u1eq} simplifies the above equation
\begin{align*}
\tta \cdot \nabla u +au&= e^{-2i \fii} (\ol{\del}u_{-1}+\del u_{-3}+au_{-2})  + e^{2i \fii} (\del u_{1} + \ol{\del}u_{3}+au_{2}) \\
& \qquad + \ol{\del } u_{1}+ \del u_{-1} +au_0.
\end{align*}
Now using \eqref{f0_defn_a>0} and \eqref{f2_defn_a>0}, we conclude \eqref{TransportEq2Tensor2}
\begin{align*}
\tta \cdot \nabla u +au &= e^{-2i \fii} f_2 + e^{2i \fii} \ol{f_2}+f_0 = \langle \bF_{\psi} \tta,  \tta \rangle.
\end{align*}

\end{proof}

\section*{Acknowledgment}
The work of O. Scherzer has been supported by the Austrian Science Fund (FWF), Project P26687-N25 
(Interdisciplinary Coupled Physics Imaging). The work of A. Tamasan has been supported by the NSF-Grant DMS 1312883.


\end{document}